\documentclass[12pt]{amsart}

\def \fg{{\mathfrak g}}

\def \fa{{\mathfrak a}}

\def \fn{{\mathfrak n}}
\def \fq{{\mathfrak q}}
\def \u{{\mathfrak u}}

\def \C{{\mathbb C}}

\def \R{{\mathbb R}}

\newtheorem{theorem}{Theorem}[section]
\newtheorem{lemma}[theorem]{Lemma}
\newtheorem{proposition}[theorem]{Proposition}

\numberwithin{equation}{section}

\begin{document}

\baselineskip=16pt

\title[A note on Iwasawa-type decomposition]{A note on Iwasawa-type decomposition}

\author[P. Foth]{Philip Foth}

\address{CEGEP Champlain St. Lawrence, Qu\'ebec, Canada G1V 4K2 \newline 
\indent Department of Mathematics, University of Arizona, Tucson, AZ 85721-0089}

\email{phfoth@gmail.com}
\subjclass{Primary 15A23, secondary 53D17, 15A18.}
\keywords{Eigenvalue, pseudo-unitary, admissible element, Iwasawa decomposition.}

\date{September 23, 2010}

\begin{abstract}
We study the Iwasawa-type decomposition of an open subset of ${\rm SL}(n, \C)$ as ${\rm SU}(p,q)AN$. We show 
that the dressing action of ${\rm SU}(p,q)$ is globally defined on the space of admissible elements in $AN$. 
We also show that the space of admissible elements is a multiplicative subset of $AN$. We establish a geometric 
criterion: the symmetrization of an admissible element maps the positive cone in $\C^n$ into itself.
\end{abstract}

\maketitle

\section{Introduction}
In Poisson geometry the groups ${\rm SU}(p,q)$ and $AN$ (the upper-triangular subgroup of ${\rm SL}(n,\C)$ with real positive diagonal entries) are naturally dual to each other \cite{FL}. Therefore, it is important to know the geometry of the orbits of the dressing action. We show that the right dressing action of ${\rm SU}(p,q)$ is globally defined on the open subset of the so-called admissible elements of $AN$ (see next section).

We also show that the admissible elements can be characterized as follows: these are exactly those elements of $AN$
whose symmetrization maps the closure of the positive cone in $\C^n$ into the positive cone. In addition, we 
establish a useful fact that the set of admissible elements is a multiplicative subset of $AN$.

\section{Admissible elements}

Let $p$ and $q$ be positive integers and let $n=p+q$. Consider the space $\C^n$ and the group $G_0={\rm SU}(p,q)$,
The group $G_0$ is the subgroup of $G={\rm SL}(n, \C)$, which preserves the following sesquilinear pairing on $\C^n$:
$$
\langle{\bf x}, {\bf y}\rangle = \sum_{i=1}^px_i{\bar y}_i - \sum_{j=p+1}^nx_j{\bar y}_j .
$$
Denote the corresponding norm by $||{\bf x}||$. A vector ${\bf x}\in\C^n$ is said to be \emph{timelike} if the square of its norm is positive, and \emph{spacelike}, if it is negative.

Let $A\subset G$ be the subgroup of positive real diagonal matrices and $N$ be the unipotent upper triangular subgroup.
Denote by $\fg$, $\fa$, $\fn$, and $\fg_0$ the Lie algebras of $G$, $A$, $N$, and $G_0$.   
The Iwasawa-type decomposition for $G$ states that for 
an open dense subset $G^{\vee}$ of $G$ one has \cite[p.167]{Vil3}:
$$
G^{\vee} = \coprod_{w\in W/W_0} G_0{\dot{w}}AN, 
$$
where $W=W(\fg, \fa)$ is the Weyl group, $W_0$ is the subgroup of $W$ with representatives in 
$K = {\rm SU}(n)\cap G_0$ and $\dot{w}$ is a representative of $w\in W/W_0$ in ${\rm SU}(n)$. 

An important question is to determine which elements of $G$ allow such a decomposition for a given choice of $w$. 
The case of particular interest is when $w=1$, since it is related to the dressing action in Poisson geometry. 

Let $J_{p,q}$ be the diagonal matrix $J_{p,q}=\rm{diag}(\underbrace{1,1,...,1}_p, \underbrace{-1,-1,...-1}_q)$. 
Introduce the following involution on the space of $n\times n$ matrices: $$A^{\dagger}=J_{p,q}A^*J_{p,q}\ ,$$ where 
$A^*$ is the usual conjugate transpose. The Lie algebra $\fg$ splits, as a vector space, into the direct sum
of $\pm 1$-eigenspaces of $\dagger$: $ \fg = \fq + \fg_0$. Let also $Q\subset G$ be the submanifold of elements   
satisfying $A^\dagger = A$. Clearly, $\exp(\fq)\subset Q$. 

The next definition is quintessential for this paper. 
We say that $${\vec{\lambda}}={\rm diag}(\lambda_1, ..., \lambda_p, \mu_1, ..., \mu_q)\subset\fa$$ 
is {\it admissible} if $\lambda_i> \mu_j$. Clearly , using the action of $W_0$, one can arrange $\lambda_i$'s and $\mu_j$'s in the non-increasing order, and the condition of admissibility will become simply $\lambda_p > \mu_1$. Same 
definition applies for $A$. Next, we say that an element $X\in\fq$ is admissible, if it is $G_0$-conjugate to 
an admissible element in $\fa$. The set of admissible elements in $\fq$ form an open cone. Define the admissible
elements in $Q$ as the exponents of those in $\fq$. Finally, an element $b\in AN$ is called admissible, if it obtained from an admissible element of $A$
by the right dressing action. 

The above definition stems from the work of Hilgert, Neeb and others, see e.g. \cite{Neeb}. Recall the definition of the right dressing action of $G_0$ on $AN$. Let $b\in AN$ and $g\in G_0$. Assume there exist $b'\in AN$ and $g'\in G_0$ such that $bg=g'b'$. In this case we write $b'=b^g$.   

One of the important properties of the set of admissible elements can be observed in the following result, which asserts that admissible elements map the timelike cone into itself. 

\begin{lemma}
Let $s\in Q$ be admissible and let ${\bf x}\in\C^n$ be timelike. Then $s{\bf x}$ is timelike as well.
\label{l:l21}
\end{lemma}

\begin{proof} Decompose $s=g^{-1}e^{\vec{\lambda}}g$, where $g\in G_0$ and 
${\vec{\lambda}}={\rm diag}(\lambda_1, ..., \lambda_p, \mu_1, ..., \mu_q)\subset\fa$
is admissible, so that we have $\lambda_i > \mu_1 \ge \mu_j$ for all $1\le i\le p$ and $2\le j\le q$.  

Let ${\bf y}=g{\bf x}$. Since the group $G_0$ preserves the cone of timelike elements, ${\bf y}$ is 
timelike as well and, denoting its coordinates
$${\bf y}=(z_1, ..., z_p, w_1, ..., w_q) ,$$ we see that 
$$
||{\bf y}||^2 = \sum_{i=1}^p |z_i|^2 - \sum_{j=1}^q |w_j|^2 = r, \ \ r\in\R, \ \ r >0 \ . 
$$

Now, let us show that $e^{\vec{\lambda}}{\bf y}$ is timelike, by using the above equation and expressing 
$|w_1|^2$ in terms of the other coordinates:
$$
\sum_{i=1}^p e^{2\lambda_i}|z_i|^2 - \sum_{j=1}^q e^{2\mu_j}|w_j|^2 = 
$$
$$
= \sum_{i=1}^p(e^{2\lambda_i}-e^{2\mu_1})|z_i|^2 + \sum_{j=2}^q(e^{2\mu_1}-e^{2\mu_j})|w_j|^2 \
+ \ re^{2\mu_1}\ >\ 0 \ .
$$
The last step in the proof is to notice that $g^{-1}$ prserves the timelike cone, and thus 
$s{\bf x}=g^{-1}e^{\vec{\lambda}}{\bf y}$ is timelike. 
\end{proof}

Later on, in Proposition \ref{p:p661} we will show that conversely, admissible elements 
can be characterized by this property.

\section{Dressing action and admissible elements}

Here we will give an excplicit computational indication
that if ${\vec{\lambda}}\in \fa$ is admissible, then the whole orbit $\exp({\vec{\lambda}})G_0$ admits a global decomposition, i.e. $\exp({\vec{\lambda}})G_0 \subset G_0AN$. Therefore, the right 
dressing action is globally defined on the set of admissible elements in $AN$. Note that this is not true for the
left dressing action, as a simple $2\times 2$ example can show. 

Next, we will use the Gram-Schmidt orthogonalization process to show that for any $g\in G_0$ and admissible $\vec{\lambda}$, we have  $\exp({\vec{\lambda}})g \subset G_0AN$. A short proof of this will be given in
the next section. Let us denote the columns of $g$ by 
${\bf w}_1$, ..., ${\bf w}_n$.  The columns of $g$ are (pseudo) orthonormal with respect to $\langle\cdot,\cdot\rangle$, and the first $p$ columns are timelike, and the last $q$ are spacelike. 
       
Denote the columns of $\exp({\vec{\lambda}})g$ by ${\bf v}_1$, ..., ${\bf v}_n$. This element of $G$ is obtained from $g$ by multiplying the first row by $e^{\lambda_1}$, ..., the $p$-th row by $e^{\lambda_p}$, the $(p+1)$-st row by 
$e^{\mu_1}$, ..., and the last row by $e^{\mu_q}$. 

An important observation is that due to the admissibility of ${\vec{\lambda}}$, the first $p$ columns of 
$\exp({\vec{\lambda}})g$ will remain timelike, however nothing definite can be said about the last $q$. 

The decomposition $\exp({\vec{\lambda}})g = sb$ with $s\in G_0$ and $b\in AN$ is an analogue of the Gram-Schmidt 
orthogonalization process for the pseudo-metric $\langle\cdot,\cdot\rangle$. 

Denote the columns of $s$ by ${\bf u}_1$, ..., ${\bf u}_n$. Proving the existence of such decomposition amounts to showing that if we follow the Gram-Schmidt process, the first $p$ columns of $s$ will be timelike, and the last $q$ will be spacelike, and that the diagonal entries of $b$ will be positive real numbers. Let us denote the diagonal entries of $b$ by $(r_1, ..., r_n)$ and the off-diagonal by $m_{ij}$, where $m_{ij}=0$ for $i > j$. 

Consider the first step of the Gram-Schmidt process, namely that $r_1=||{\bf v}_1||$ and 
${\bf u}_1 = {\bf v}_1/r_1$. Since ${\bf v}_1$ is timelike, we see that $r_1$ is real positive and that
${\bf u}_1$ is timelike. 

Now we move to the second step, (obviously, under the assumption that $p>1$), 
which we consider in detail, because it lays the foundation for the other columns: 
$$
{\bf v}_2 = m_{12}{\bf u}_1 + r_2{\bf u}_2. 
$$ 
Here $m_{12}=\langle {\bf v}_2, {\bf u}_1\rangle$ and $r_2 = ||{\bf v}_2 - m_{12}{\bf u}_1||$. In order to complete
this step we need to show that the vector ${\bf v}_2 - m_{12}{\bf u}_1$ is timelike. 

Note that 

$$
||{\bf v}_2 - m_{12}{\bf u}_1||^2 = ||{\bf v}_2||^2 - |\langle {\bf v}_2, {\bf u}_1\rangle |^2 , 
$$
so we just need to show that this number is positive. 
Since $||{\bf u}_1||^2 = 1$, this is equivalent to showing that 
$$
||{\bf v}_2||^2\cdot ||{\bf u}_1||^2 > |\langle {\bf v}_2, {\bf u}_1\rangle |^2 , 
$$
or, after multiplying both sides by $r_1^2$, that
\begin{equation}
||{\bf v}_2||^2\cdot ||{\bf v}_1||^2 > |\langle {\bf v}_2, {\bf v}_1\rangle |^2 
\label{eq:ine2}
\end{equation}

Denote the coordinates of the 
vector ${\bf w}_1$ by $(a_1, .., a_p, b_1, ..., b_q)$, and the coordinates of 
${\bf w}_2$ by $(c_1, .., c_p, d_1, ..., d_q)$. 
We have:
\begin{equation}
\sum_{i=1}^p |a_i|^2- \sum_{j=1}^q |b_j|^2 = \sum_{i=1}^p |c_i|^2- \sum_{j=1}^q |d_j|^2 =1
\label{eq:norm1}
\end{equation}
and 
\begin{equation}
\sum_{i=1}^p a_i{\bar c}_i - \sum_{j=1}^q b_j{\bar d}_j = 0. 
\label{eq:dot0}
\end{equation}

The coordinates of the vector ${\bf v}_1$ are given by 
$$
(e^{\lambda_1}a_1, ..., e^{\lambda_p}a_p, e^{\mu_1}b_1, ..., e^{\mu_q}b_q).
$$
and ${\bf v}_2$ by 
$$
(e^{\lambda_1}c_1, ..., e^{\lambda_p}c_p, e^{\mu_1}d_1, ..., e^{\mu_q}d_q).
$$

The RHS of Equation (\ref{eq:ine2}) now becomes, using (\ref{eq:norm1}):
$$
\left( \sum_{i=1}^p e^{2\lambda_i}|a_i|^2 - \sum_{j=1}^q e^{2\mu_j}|b_j|^2 \right)
\cdot
\left( \sum_{i=1}^p e^{2\lambda_i}|c_i|^2 - \sum_{j=1}^q e^{2\mu_j}|d_j|^2 \right)
=
$$

$$
=\left( 
\sum_{i=1}^p (e^{2\lambda_i}-e^{2\mu_1})|a_i|^2 +
\sum_{j=2}^q (e^{2\mu_1}-e^{2\mu_j})|b_j|^2 + e^{2\mu_1}
\right) \cdot
$$

$$\cdot\left( 
\sum_{i=1}^p (e^{2\lambda_i}-e^{2\mu_1})|c_i|^2 +
\sum_{j=2}^q (e^{2\mu_1}-e^{2\mu_j})|d_j|^2 + e^{2\mu_1}
\right) ,
$$
which is strictly greater than
$$
\left( 
\sum_{i=1}^p (e^{2\lambda_i}-e^{2\mu_1})|a_i|^2 +
\sum_{j=2}^q (e^{2\mu_1}-e^{2\mu_j})|b_j|^2
\right) \cdot\left( 
\sum_{i=1}^p (e^{2\lambda_i}-e^{2\mu_1})|c_i|^2 +
\sum_{j=2}^q (e^{2\mu_1}-e^{2\mu_j})|d_j|^2
\right) ,
$$
which in turn, by Cauchy-Schwarz, is greater or equal than 
$$
|\sum_{i=1}^p (e^{2\lambda_i}-e^{2\mu_1})a_i{\bar c}_i + 
\sum_{j=2}^q (e^{2\mu_1}-e^{2\mu_j})b_j{\bar d}_j |^2, 
$$
which is exactly $|\langle {\bf v}_1, {\bf v}_2 \rangle |^2$, if we take into account
(\ref{eq:dot0}). 

Now we will similarly show that ${\bf u}_k$ is timelike for any $k \le p$. We have 
$$
{\bf v}_k = \sum_{\ell=1}^{k-1} m_{\ell k}{\bf u}_\ell + r_k {\bf u}_k, 
$$
where $m_{\ell k} = \langle {\bf v}_k, {\bf u}_\ell\rangle$. 

This amounts to showing that
$$
||{\bf v}_k||^2 > \sum_{\ell=1}^{k-1} |\langle {\bf v}_k , {\bf u}_{\ell} \rangle | ^2  
$$ 

Consider the subspace $V\subset \C^n$ spanned by ${\bf u}_1, ..., {\bf u}_{k-1}$, which is the 
same subspace that is spanned by ${\bf v}_1$, ..., ${\bf v}_{k-1}$. It is isomorphic to the subspace
$W$ spanned by ${\bf w}_1$, ..., ${\bf w}_{k-1}$, and the explicit isomorphism is given
by multiplying the first coordinate by $e^{-\lambda_1}$, ..., the $p$-th by $e^{-\lambda_p}$, 
the $(p+1)$-st by $e^{-\mu_1}$, ...., and the last by $e^{-\mu_q}$. 

Denote by ${\bf y}_k= \sum_{\ell=1}^{k-1}m_{\ell k}{\bf u}_{\ell}$
If $\langle {\bf v}_k,  {\bf y}_k \rangle =0$, then by definition of ${\bf y}_k$ and $m_{\ell k}$ 
this would imply that ${\bf v}_k$ is orthogonal to all the ${\bf u}_i$ for $1\le i\le k-1$, and thus 
all $m_{\ell k} = 0$ and, similar to the case $k=1$, the vector 
${\bf u}_k = {\bf v}_k/|{\bf v}_k|$ is timelike as desired. 

Thus we can assume that 
$$ 
\alpha_k = \langle {\bf v}_k,  {\bf y}_k \rangle = \sum_{\ell=1}^{k-1} |m_{\ell k}|^2 > 0 \ .
$$
 
Define 
$$
{\bf u} = \frac{{\bf y}_k}{\sqrt{\alpha_k}}\ .
$$
This is an essentially unique element of $V$ (up to a circle factor) with the property that: 
$$
\langle {\bf v}_k, {\bf u} \rangle \cdot {\bf u} = \sum_{\ell=1}^{k-1} m_{\ell k}{\bf u}_\ell .
$$
Let also ${\bf w}$ be the unit vector in $W$, which is the rescaled image of ${\bf u}$ under the above
homomorphism. The previous proof for the case $k=2$ now applies verbatim, with the role 
of ${\bf w}_1$ played by ${\bf w}$ and the role of ${\bf v}_2$ played by ${\bf v}_k$. 

Thus we have established that the first $p$ columns of $s$ are timelike vectors, and that $r_1$, ..., $r_p$
are positive real numbers. Due to the Sylvester's law of inertia for quadratic forms, 
the only way to complete this to a (pseudo) orthonormal basis 
with respect to $\langle\cdot,\cdot\rangle$ is to add $q$ spacelike vectors - therefore if we continue the 
Gram-Schmidt orthogonalization process, then the last $q$ columns of $s$ will be spacelike as required.

\medskip

\section{Iwasawa-type decomposition}

Let us introduce more notation. Denote by  $Q_{\rm adm}$ and $\fq_{\rm adm}$ the sets of admissible elements in 
$Q$ and $\fq$ respectively. Recall that $Q_{\rm adm}=\exp(\fq_{\rm adm})$. The following result is straightforward. 

\begin{lemma} On these sets of admissible elements
the map $\exp$ is a diffeomorphism.
\end{lemma} 

Now consider the symmetrization map: 
$$
{\rm Sym}:\ AN\to Q, \ \ Y\mapsto Y^{\dagger}Y \ .
$$ 
This map sends the orbits of the dressing action to the conjugation orbits by the action of $G_0$ and 
therefore maps admissible elements to admissible. In fact, on the set of admissible elements, this map
is, again, bijective, and, moreover, a simple computation of the differential can show that this 
is a diffeomorphism. 

Let, as before, $\vec{\lambda}\in\fa$ be admissible, let $g_0\in G_0$, and let us find an 
explicit decomposition
$$
\exp(\vec{\lambda})g_0 = g_0'an
$$ 
with $g_0'\in G_0$, $a\in A$, and $n\in N$. Notice that the symmetrization map yields: 
$$
g_0^{-1}\exp(2\vec{\lambda})g_0 = n^{\dagger}a^2 n\ .
$$
The right hand side provides the Gauss decomposition 
(also known as the triangular or the LDV-factorization) of the left 
hand side. It exists if and only if the leading principal minors of the matrix on the left are non-zero. 
However, it was shown in \cite[Proposition 4.1]{Fot}, that the eigenvalues of the principal minors of 
an admissible element satisfy certain interlacing conditions, similar to the Gelfand-Tsetlin conditions 
in the Hermitian symmetric case. In particular, since all the eigenvalues of $s\in Q_{\rm adm}$ are positive,
then the eigenvalues of any leading principal minor would be positive as well, and therefore the Gauss 
decomposition exists. 

The diagonal entries of $a^2$ are the ratios of the leading principal minors:

$$
(a^2)_{ii} = \frac{\Delta_i}{\Delta_{i-1}}\ .
$$  

Thus we also see that the entries of $a^2$ are positive, as required. Thus, we have established:

\begin{proposition}
If ${\vec{\lambda}}\in \fa$ is admissible, then the whole orbit $\exp({\vec{\lambda}})G_0$ admits a global decomposition, i.e. $\exp({\vec{\lambda}})G_0 \subset G_0AN$.
\end{proposition}

Another interesting property of $(AN)_{\rm adm}$ is that it is a multiplicative set:

\begin{proposition}
If two elements $b_1$ and $b_2$ from $AN$ are admissible, $b_1, b_2\in (AN)_{\rm adm}$, then their
product $b_1b_2$ is admissible as well.
\end{proposition}  

\begin{proof}
Applying the dressing action, we can actually assume that $b_2=a\in A_{\rm adm}$. This follows from 
the fact that for $g\in G_0$ and $b_1, b_2\in AN$:
$$
g^{b_1b_2}=(g^{b_2})^{b_1} \ \ \ {\rm and} \ \ \ 
(b_1b_2)^g = b_1^{(g^{b_2})}b_2^g\ .
$$
Applying the symmetrization map to $b_1a$, we obtain a new matrix $f=ab_1^{\dagger}b_1a\in Q$, which 
we need to prove admissible. 

Since the space $A_{\rm adm}$ is clearly connected and $a$ is admissible, consider a path $a(t)$ 
such that $a(0)={\rm Id}$, $a(1)=a$ and $a(t)\in A_{\rm adm}$ for $0<t\leq 1$. Also denote 
$f(t)=a(t)b_1^{\dagger}b_1a(t)$. 

Assume, on the contrary, that $f(1)=f$ is not admissible and let ${\mathcal{E}}\subset [0,1]$ be the 
(non-empty) subset defined by the property that $f(t)$ for $t\in{\mathcal{E}}$ is not admissible. 
Then consider $\tau=\inf{\mathcal{E}}$. 

We recall \cite{Neeb} that the space $\fq_{\rm adm}$ forms a convex cone in $q$, therefore a simple infinitesimal computation can show that $Q_{\rm adm}$ is a connected open subset of $Q$ and 
any element from the boundary of $Q_{\rm adm}$, such as  
$f(\tau)$, has the property that its eigenvalues remain real and, moreover, the 
lowest eigenvalue $\lambda_p$ corresponding to the timelike cone and the highest eigenvalue $\mu_1$ corresponding 
to the spacelike cone collide: $\lambda_p=\mu_1=\beta$. Thus, there exists a whole plane of eigenvectors
containing vectors from both $\C^n_+$ and $\C^n_-$. Thus it must contain an eigenvector, which we denote by 
${\bf z}$, from the nullcone. 

For this eigenvector we have the following equation: $f(\tau){\bf z}=\beta{\bf z}$. Or, equivalently,
$$
a(\tau)(b_1^{\dagger}b_1)a(\tau){\bf z}=\beta{\bf z}\ .
$$
Note that both $a(\tau)$ and $b_1^{\dagger}b_1$, being admissible, map the timelike cone $\C^n_+$ into itself
by Lemma \ref{l:l21}. Similar proof shows that they map the nullcone minus the origin 
$\C^n_0\setminus\{ 0\}$ inside the timelike cone. Therefore, 
the element $a(\tau)(b_1^{\dagger}b_1)a(\tau)$ also maps $\C^n_0\setminus\{ 0\}$ inside of $\C^n_+$ and 
thus ${\bf z}$ cannot be its eigenvector. \end{proof}

In geometric terms, we have established the following characterization of the space of admissible elements 
$Q_{\rm adm}$: 

\begin{proposition}
An element $s\in Q$ is admissible if and only if it has real positive
eigenvalues and maps the closure of the 
timelike cone $\overline{\C^n_+}=\C^n_+\cup\C^n_0$ into the timelike cone $\C^n_+$ (plus the origin).
\label{p:p661}
\end{proposition}

Notice that since $a$ from the above proof is admissible 
diagonal, the pseudo-hermitian analogue of the Rayleigh-Ritz ratio 
$$
{\mathcal{R}}_A = \frac{{\bf x}^\dagger A{\bf x}}{{\bf x}^\dagger {\bf x}}
$$
for the matrices $f$ and $s=b_1^{\dagger}b_1$ will have the following properties. 
Let ${\bf x}\in\C^n_+$ be timelike, then, as we saw earlier, $a{\bf x}$ is timelike as well
with a bigger norm by Lemma \ref{l:l21}. 

Thus, for ${\bf x}\in\C^n_+$ we have:
$$
{\mathcal{R}}_{asa}({\bf x})  = \frac{{\bf x}^{\dagger}asa{\bf x}}{{\bf x}^\dagger{\bf x}} >
\frac{(a{\bf x})^{\dagger}s(a{\bf x})^\dagger}{(a{\bf x})^{\dagger}a{\bf x}}={\mathcal{R}}_s(a{\bf x})\ .
$$
It follows since $a$ maps $\C^n_+$ to itself, that 
$$
\min_{{\bf x}\in\C^n_+}{\mathcal{R}}_s({\bf x}) \le \min_{{\bf x}\in\C^n_+}{\mathcal{R}}_s(a{\bf x})
< \min_{{\bf x}\in\C^n_+}{\mathcal{R}}_{asa}({\bf x}) 
$$
From \cite[Theorem 2.1]{F2} it follows that $\lambda_1(asa) > \lambda_1(s)$. 
Similar inequalities can be established for other eigenvalues.

The notion of admissibility can be extended to the whole group $G={\rm SL}(n, \C)$. 
We say that $g\in G$ is admissible, if it decomposes as $g=hb$, with $h\in G_0$ and an admissible 
$b\in (AN)_{\rm adm}$. Such a decomposition, if exists, is clearly unique. Note, moreover, that 
$g^{\dagger}g= b^{\dagger}b$. The singular admissible spectrum of $g$ are the square roots of the 
spectrum of $g^{\dagger}g$. (Note of warning: one should not define $Q_{\rm adm}$ by inducing
this definition from $G$, but rather as we did previously.)

Suppose that $g_1$, $g_2$ are two admissible elements from $G$ such that $g_1=h_1b_1$ and 
$g_2=h_2b_2$. Then we have $g_1g_2=h_1b_1h_2b_2=h_1h_2^{b_1}b_1^{h_2}b_2$, where the powers mean the
dressing actions. This is possible, because $b_1$ is admissible. Thus we have established that the
product $g_1g_2$ is admissible if and only if the product $b_1^{h_2}b_2$ is such.

\section{Example of the group SU(1,1)}

The Poisson geometry related to the group ${\rm SU}(1,1)$ was considered in detail in \cite{FM}. Here
we just recall several facts to illustrate the results of this paper. 
First of all, the space $\fq_{adm}$ is the following convex cone of matrices 
$$
\left( \begin{array}{cc}
z &  x+iy\\
-x+iy &  - z
\end{array} \right)\ : \ \ x,y,z\in\R, \ \ \ z^2-x^2-y^2 > 0 ,\ \ z > 0. 
$$
Next, we define $Q_{\rm adm}=\exp(\fq_{\rm adm})$. Consider an element 
$$
\left( \begin{array}{cc}
t_1 &  m\\
-{\bar m}  &  t_2
\end{array} \right) \ \ \in \ \ Q\ ,
$$
where $t_1,t_2\in\R$, and $m\in\C$ satisfy the determinant condition $t_1t_2+|m|^2=1$. It is admissible if and only 
if its eigenvalues are real and positive, and, moreover, the eigenvalue for the timelike cone 
is greater than the other one. 
This translates into the following two conditions on the coefficients of this matrix:
$$
t_1+t_2 > 2 \ \ \ {\rm and} \ \ \ t_1 > 1\ .
$$
Next, an element 
$$
\left( \begin{array}{cc} r & n  \\ 0 & r^{-1} \end{array} \right) \ \in \ AN
$$ 
is admissible if and only if 
$$r>1 \ \ \  {\rm and} \ \ \ r^2+r^{-2}-|n|^2>2\ .\
$$ 

Now, let us write down the Iwasawa-type decomposition $G\supset G_0AN$ explicitly. 
Let us consider a general element of ${\rm SL}(2, \C)$:
$$
\left( \begin{array}{cc}
a &  b\\
c & d
\end{array} \right) \ :\ \ \ a,b,c,d,\in\C, \ \ ad-bc=1\ .
$$
The condition that this general element admits such a decomposition is simply $|a|>|c|$:
$$
\left( \begin{array}{cc}
a &  b\\
{} & {} \\
c & d
\end{array} \right)
=
\left( \begin{array}{cc}
\frac{a}{\sqrt{|a|^2-|c|^2}} &  \frac{\bar c}{\sqrt{|a|^2-|c|^2}}\\
{} & {} \\
\frac{c}{\sqrt{|a|^2-|c|^2}} & \frac{\bar a}{\sqrt{|a|^2-|c|^2}}
\end{array} \right)
\cdot
\left( \begin{array}{ccc}
\sqrt{|a|^2-|c|^2} & {} & \left( b-\frac{\bar c}{|a|^2-|c|^2} \right) \frac{\sqrt{|a|^2-|c|^2}}{a}\\
{} & {} & {} \\ 
0 & {} & \frac{1}{\sqrt{|a|^2-|c|^2}}
\end{array} \right)
$$
Thus, it is quite easy to see that for any element $g\in G$ and $\exp({\vec{\lambda}})\in A_{\rm adm}$, 
the element $\exp({\vec{\lambda}})g\in G_0AN$. Explicitly, if 
$$
g=
\left( \begin{array}{cc}
u & v \\
{\bar v} & {\bar u}
\end{array} \right)\ , \ \ {\rm and} \  \ 
\exp({\vec{\lambda}}) = 
\left( \begin{array}{cc}
e^{\lambda} & 0 \\ 0 & e^{-\lambda}
\end{array}\right)\ , \ \ 
\lambda>0\ ,
$$
then
$$ 
\exp({\vec{\lambda}})g=
\left( \begin{array}{cc}
e^{\lambda}u & e^{\lambda}v \\
e^{-\lambda}{\bar v} & e^{-\lambda}{\bar u}
\end{array} \right)
$$
is an admissible element of $G$. 

The statemement that the admissibility of two elements $b_1$ and $b_2$ of $AN\subset {\rm SL}(2, \C)$ implies the admissibility of their product is also a short computational affair.

\end{document}